\documentclass[12pt]{amsart}

\newtheorem{theorem}{Theorem}[section]
\newtheorem{proposition}[theorem]{Proposition}
\newtheorem{lemma}[theorem]{Lemma}
\newtheorem{corollary}[theorem]{Corollary}

\theoremstyle{definition}

\newtheorem{example}[theorem]{Example}

\newtheorem{question}[theorem]{Question}
\newtheorem{conjecture}[theorem]{Conjecture}
\newtheorem{remark}[theorem]{Remark}

\newcommand{\ZZ}{ \ensuremath{\mathbb{Z}}}

\newcommand{\lk}{{\mathrm{lk}}}
\newcommand{\st}{\mathrm{st}}
\newcommand{\kk}{\mathbf{k}}

\newcommand{\mideal}{\ensuremath{\mathfrak{m}}}

\def\cocoa{{\hbox{\rm C\kern-.13em o\kern-.07em C\kern-.13em o\kern-.15em A}}}

\newcommand{\Soc}{\mathrm{Soc}\ \!}

\begin{document}

\title[On $r$-stacked triangulated manifolds]{On $r$-stacked triangulated manifolds}

\author{Satoshi Murai}
\address{
Satoshi Murai,
Department of Mathematical Science,
Faculty of Science,
Yamaguchi University,
1677-1 Yoshida, Yamaguchi 753-8512, Japan.
}

\author{Eran Nevo}
\address{
Eran Nevo,
Department of Mathematics,
Ben Gurion University of the Negev,
Be'er Sheva 84105, Israel
}


\thanks{
Research of the first author was partially supported by KAKENHI 22740018.
Research of the second author was partially supported by Marie Curie grant IRG-270923 and by ISF grant.
}

\begin{abstract}
The notion of $r$-stackedness for simplicial polytopes was introduced by McMullen and Walkup in 1971 as a generalization of stacked polytopes.
In this paper,
we define the $r$-stackedness for triangulated homology manifolds
and study their basic properties.
In addition, we find a new necessary condition for face vectors of triangulated manifolds
when all the vertex links are polytopal.
\end{abstract}

\maketitle

\section{Introduction}

A triangulated $d$-ball is said  to be {\em $r$-stacked} if it has no interior faces of dimension $\leq d-r-1$,
and the boundary of an $r$-stacked $d$-ball is called an $r$-stacked $(d-1)$-sphere.
It is known that $r$-stacked $d$-balls and $(d-1)$-spheres with $r < \frac d 2$ have many nice combinatorial properties,
and they have been used to obtain several important results on polytopes and triangulated spheres.
For example, they appeared in Barnette's lower bound theorem \cite{Ba, Ba2} and in the generalized lower bound conjecture given by McMullen and Walkup \cite{MW}.
They also appeared in the proof of the sufficiency of the famous $g$-theorem by Billera and Lee \cite{BL} (see \cite{KlL})
as well as in the construction of many non-polytopal triangulated spheres given by Kalai \cite{Ka1}.
The purpose of this paper is to extend this notion to triangulated manifolds,
and establish their fundamental properties.

Throughout the paper,
we fix a field $\kk$.
For a simplicial complex $\Delta$ and its face $F \in \Delta$,
the {\em link of $F$ in $\Delta$} is the simplicial complex
$$\lk_\Delta(F)=\{ G \in \Delta: F \cup G  \in \Delta \mbox{ and } F \cap G = \emptyset\}.$$
A simplicial complex $\Delta$ of dimension $d$ is said to be a {\em $(\kk$-$)$homology $d$-sphere}
if, for all faces $F \in \Delta$ (including the empty face $\emptyset$),
one has $\beta_i(\lk_\Delta(F))=0$ for $i \ne d-\#F$ and $\beta_{d-\#F}(\lk_\Delta(F))=1$,
where $\beta_i(\Delta)=\dim_\kk \widetilde H_i(\Delta;\kk)$ is the $i$th {\em Betti number} of $\Delta$ over $\kk$.
A simplicial complex is said to be {\em pure} if all its facets have the same dimension.
A {\em ($\kk$-)homology $d$-manifold without boundary} is a $d$-dimensional pure simplicial complex
all whose vertex links are $\kk$-homology spheres.
A pure $d$-dimensional simplicial complex $\Delta$ is said to be a {\em ($\kk$-)homology $d$-manifold with boundary} if it satisfies
\begin{itemize}
\item[(i)] for all $\emptyset \ne F \in \Delta$,
$\beta_i(\lk_\Delta(F))$ vanish for $i \ne d - \#F$ and is equal to $0$ or $1$ for $i = d - \#F$.
\item[(ii)] the {\em boundary} $\partial \Delta=\{F \in \Delta: \beta_i(\lk_\Delta(F))=0\} \cup \{ \emptyset \}$ of $\Delta$
is a $\kk$-homology $(d-1)$-manifold without boundary.
\end{itemize}
Triangulations of topological manifolds are examples of homology manifolds.
Also, condition (ii) can be omitted if we replace $\kk$ by $\ZZ$ (see \cite{Mi}).

We say that a homology $d$-manifold $\Delta$ with boundary is {\em $r$-stacked} if it has no interior faces (namely, faces which are not in $\partial \Delta$) of dimension $\leq d-r-1$.
Also, a homology manifold without boundary is said to be {\em $r$-stacked} if it is the boundary of an $r$-stacked homology manifold with boundary.
We prove the following properties for $r$-stacked homology manifolds.

\begin{itemize}
\item[(a)] {\em Enumerative criterion:}
We give a simple criterion for the $r$-stackedness in terms of $h$-vectors and Betti numbers
for homology manifolds with boundary
(Theorem \ref{3.1}).
Also, we give a similar result for $(r-1)$-stacked homology $(d-1)$-manifolds without boundary with $r \leq \frac d 2$
when all the vertex links are polytopal (Corollary \ref{6.9}).
In particular, these results prove that $r$-stackedness depends only on face numbers and Betti numbers for these manifolds.
\item[(b)] {\em Vanishing of Betti numbers and missing faces:}
We show that if a homology manifold (with or without boundary) is $r$-stacked,
then it has zero Betti numbers and no missing faces in certain dimensions
(Corollary \ref{3.3} and Theorem \ref{4.1}).
\item[(c)] {\em Uniqueness of stacked manifolds:}
For an $(r-1)$-stacked $(d-1)$-manifold $\Delta$ without boundary,
it is shown that if $r \leq \frac d 2$ then there is a unique $(r-1)$-stacked homology manifold $\Sigma$
such that $\partial \Sigma=\Delta$ (Theorem \ref{4.3}).
\item[(d)] {\em Local criterion:}
For $r < \frac d 2$,
we show that a homology $(d-1)$-manifold without boundary is $(r-1)$-stacked if and only if all its vertex links are $(r-1)$-stacked (Theorem \ref{4.4}).
\item[(e)] {\em The $\tilde g$-vector --- a new necessary condition for face vectors:}
Motivated by a recent conjecture given by Bagchi and Datta,
we define the $\tilde g$-vector of a simplicial complex $\Delta$,
and show that it is an $M$-vector if $\Delta$ is an $(r-1)$-stacked homology $(d-1)$-manifolds without boundary when $r \leq \frac d 2$.
Moreover, regardless of stackedness,
we show that the same result holds for connected orientable rational homology manifolds all whose vertex links are polytopal (Theorem \ref{5.4}).
\end{itemize}
Most of the results listed above are natural extensions of known results for triangulated balls and spheres.
However their proofs are not straightforward and we believe that these properties are useful in the study of face numbers of triangulated manifolds.
Indeed, 
the results about the $\tilde g$-vector
prove that \cite[Conjecture 1.6]{BD2}
holds for all homology manifolds all whose vertex links are polytopal.

About (c) and (d), the same results were proved independently by Bagchi and Datta \cite[Theorem 2.19]{BD3} with essentially the same proof.
Their results also prove vanishing of missing faces in (b).

This paper is organized as follows.
In Section 2, we recall basic properties of $h'$- and $h''$-vectors
which play an important role in the study of face numbers of homology manifolds.
In Section 3, we study $r$-stacked homology manifolds with boundary.
In Sections 4 and 5, we study $r$-stacked homology manifolds without boundary and consider the $\tilde g$-vector.

\section{ $h'$- and $h''$-vectors}

In this section, we recall $h'$- and $h''$-vectors and their algebraic meanings.
We first recall some basics on simplicial complexes.
A simplicial complex $\Delta$ on the vertex set $V$ is a collection of subsets of $V$
satisfying that $F \in \Delta$ and $G \subset F$ imply $G \in \Delta$.
Elements of $\Delta$ are called {\em faces} of $\Delta$
and subsets of $V$ which are not faces of $\Delta$ are called {\em non-faces} of $\Delta$.
The maximal faces of $\Delta$ (with respect to inclusion) are called the {\em facets} of $\Delta$
and the minimal non-faces of $\Delta$ are called the {\em missing faces} of $\Delta$.
The dimension of a face (or a missing face) $F$ is $\#F-1$,
where $\#X$ denotes the cardinality of a finite set $X$,
and a face (or a missing face) of dimension $k$ is called a {\em $k$-face} (or a {\em missing $k$-face}). 
Also, the dimension of a simplicial complex is the maximum dimension of its faces. 
For a simplicial complex $\Delta$ of dimension $d-1$,
let $f_k=f_k(\Delta)$ be the number of $k$-faces of $\Delta$
for $k=-1,0,\dots,d-1$, where $f_{-1}=1$.
The vector $f(\Delta)=(f_{-1},f_0,\dots,f_{d-1})$
is called the $f$-vector of $\Delta$.
Also, the $h$-vector $h(\Delta)=(h_0(\Delta),h_1(\Delta),\dots,h_d(\Delta))$ of $\Delta$ is
defined by the relation
$$
\sum_{i=0}^d h_i(\Delta) t^i = \sum_{i=0}^d f_{i-1}(\Delta) t^i (1-t)^{d-i}.
$$
Now we define $h'$- and $h''$-vectors.
For a simplicial complex $\Delta$ of dimension $d-1$,
its {\em $h'$-vector} $h'(\Delta)=(h'_0(\Delta),\dots,h'_d(\Delta))$ and
its {\em $h''$-vector} $h''(\Delta)=(h''_0(\Delta),\dots,h''_d(\Delta))$
are defined by
$$h'_i(\Delta)= h_i(\Delta) - \binom d i \sum_{k=1}^{i-1} (-1)^{i-k} \beta_{k-1}(\Delta)$$
for $i=0,1,\dots,d$, and by
$$h''_i(\Delta)= h_i(\Delta) - \binom d i \sum_{k=1}^{i} (-1)^{i-k} \beta_{k-1}(\Delta)
= h_i'(\Delta) - \binom d i \beta_{i-1}(\Delta)$$
for $i=0,1,\dots,d-1$ and $h''_d(\Delta)=h'_d(\Delta)$.
Note that 
$$\textstyle{h'_d(\Delta)=\sum_{\ell=0}^d (-1)^{\ell-d} f_{\ell-1}- \sum_{k=0}^{d-1} (-1)^{d-k} \beta_{k-1}(\Delta)=
\beta_{d-1}(\Delta).}$$
If one knows the Betti numbers of $\Delta$, then knowing $h(\Delta)$
is equivalent to knowing $h'(\Delta)$ (or $h''(\Delta)$).

$h'$- and $h''$-vectors have nice algebraic meanings in terms of Stanley--Reisner rings.
Let $S=\kk[x_1,\dots,x_n]$ be a polynomial ring over a field $\kk$ with $\deg x_i=1$ for all $i$.
For a simplicial complex $\Delta$ on $[n]=\{1,2,\dots,n\}$,
the {\em Stanley--Reisner ring} of $\Delta$ is the quotient ring
$$\kk [\Delta]=S/I_\Delta$$
where $I_\Delta=(x_{i_1} \cdots x_{i_k}: \{i_1,\dots,i_k\} \not \in \Delta)$.
If $\Delta$ has dimension $d-1$ and $\kk$ is infinite,
there is a sequence $\Theta=\theta_1,\dots,\theta_d \in S_1$ of linear forms such that $\dim_\kk (S/(I_\Delta+(\Theta))) < \infty$.
This sequence $\Theta$ is called a {\em linear system of parameter} (l.s.o.p.\ for short) of $\kk[\Delta]$.
In the rest of this paper, we always assume that $\kk$ is infinite.

A simplicial complex $\Delta$ of dimension $d-1$ is said to be {\em Cohen--Macaulay} (over $\kk$)
if, for all $F \in \Delta$, $\widetilde H_i(\lk_\Delta(F); \kk)$ vanishes for $i \ne d -1- \#F$.
Note that any Cohen--Macaulay simplicial complex is pure.
A pure simplicial complex is said to be {\em Buchsbaum} (over $\kk$) if all its vertex links are Cohen--Macaulay.
Homology manifolds are examples of Buchsbaum simplicial complexes.

Let $\mideal=(x_1,\dots,x_n)$ be the graded maximal ideal of $S$.
For a graded $S$-module $N$,
let $F_N(t)= \sum_{i \in \ZZ} (\dim_\kk N_i) t^i$ be the {\em Hilbert Series} of $N$,
where $N_i$ is the graded component of $N$ of degree $i$,
and let $\Soc(N)=\{ f \in N: \mideal f=0\}$ be the {\em socle} of $N$.
The following  results shown in \cite[p.\ 137]{Sc} and \cite[Theorem 3.5]{NS1}
give algebraic meanings of $h'$- and $h''$-vectors.

\begin{lemma}
\label{2.1}
Let $\Delta$ be a Buchsbaum simplicial complex of dimension $d-1$,
$\Theta=\theta_1,\dots,\theta_d$ an l.s.o.p.\ of $\kk[\Delta]$ and $R=S/(I_\Delta+(\Theta))$.
Then
\begin{itemize}
\item[(i)] (Schenzel) $F_R(t)= h'_0(\Delta)+h_1'(\Delta)t+ \cdots + h_d'(\Delta)t^d$.
\item[(ii)] (Novik--Swartz) $\dim_\kk (\Soc(R))_i \geq \binom d i \beta_{i-1}(\Delta)$ for all $i$.
In particular, there is an ideal $N \subset \Soc(R)$ such that
$F_{R/N}(t)= h''_0(\Delta)+h_1''(\Delta)t+ \cdots + h_d''(\Delta)t^d$.
\end{itemize}
\end{lemma}

In the rest of this section,
we study the relation between the vanishing of $h''$-numbers and missing faces.
For a homogeneous ideal $I \subset S$,
let $\mu_k(I)$ be the number of elements of degree $k$ in a minimal generating set of $I$,
namely, $\mu_k(I)=\dim_\kk (I/\mideal I)_k$.
Since missing faces of $\Delta$ correspond to the minimal generators of $I_\Delta$,
$\mu_k(I_\Delta)$ is equal to the number of missing $(k-1)$-faces of $\Delta$.

\begin{lemma}
\label{2.2}
Let $I \subset S$ be a homogeneous ideal, $w \in S_1$ a linear form and $k \geq 2$ an integer.
If the multiplication $\times w : (S/I)_{k-1} \to (S/I)_k$ is injective then $\mu_k(I)= \mu_k(I+(w))$.
\end{lemma}

\begin{proof}
It is clear that $\mu_k(I) \geq \mu_k(I+(w))$ for $k \geq 1$
even without injectivity assumption.
We show $\mu_k(I) \leq \mu_k(I+(w))$.
Let $\sigma_1,\dots,\sigma_t \in I$ be  elements of degree $k$ which are linearly independent in $I/\mideal I$.
What we must prove is that they are also linearly independent in $(I+(w))/\mideal(I+(w))$.

Let $\tau =\lambda_1 \sigma_1+ \cdots + \lambda_t \sigma_t \in \mideal(I+(w))$, where $\lambda_1,\dots,\lambda_t \in \kk$.
We claim $\tau \in \mideal I$.
Indeed, if $\tau \not \in \mideal I$ then there are $\rho'\in \mideal I$ and $\rho'' \not \in I$ such that
$\tau=\rho'+w \rho''$,
which implies $\rho''$ is in the kernel of the multiplication map
$ \times w : (S/I)_{k-1} \to (S/I)_k$, contradicting the assumption.
\end{proof}

\begin{lemma}
\label{2.3}
For a homogeneous ideal $I \subset S$,
if $(S/I)_j = 0$ for some $j \geq 0$ then $\mu_k(I)=0$ for $k \geq j+1$.
\end{lemma}

\begin{proof}
Since $(S/I)_j=0$, we have $I_k = \mideal_k$ for $k \geq j$.
Thus $\mu_k(I)=\mu_k(\mideal)=0$ for $j \geq k+1$.
\end{proof}

The following statement appears in \cite[Corollary 2.5 and Theorem 4.3]{Sc}.

\begin{lemma}[{Schenzel}]
\label{2.4}
Let $\Delta$ be a Buchsbaum simplicial complex of dimension $d-1$,
$R=\kk [\Delta]$,
$\Theta=\theta_1,\dots,\theta_d$ an l.s.o.p.\ of $\kk [\Delta]$,
and let $\mathcal K(i)$ be the kernel of
$$\times \theta_i: R/(\theta_1,\dots,\theta_{i-1})R \to  R/(\theta_1,\dots,\theta_{i-1})R.$$
Then 
$\dim_\kk \mathcal K(i)_j=  \binom {i-1} j \beta_{j-1}(\Delta)$ for all $i$ and $j$.
\end{lemma}

\begin{proposition}
\label{2.5}
Let $\Delta$ be a Buchsbaum simplicial complex of dimension $d-1$.
If $h_r''(\Delta)=0$ then
\begin{itemize}
\item[(i)] $\beta_k(\Delta)=0$ for $k \geq r$.
\item[(ii)] $\Delta$ has no missing faces of dimension $ \geq r+1$.
\end{itemize}
\end{proposition}

\begin{proof}
Let $\Theta$ be an l.s.o.p.\ of $\kk[\Delta]$.
Since $h_r''(\Delta)=0$, by
Lemma \ref{2.1}(ii)
all elements in
$S/(I_\Delta+(\Theta))$ of degree $r$ are contained in the socle of $S/(I_\Delta+(\Theta))$.
This fact implies
\begin{align}
\label{2-1}
S/(I_\Delta+(\Theta))_k=0\ \ \mbox{ for all }k \geq r+1.
\end{align}
Then since $\dim_\kk (S/(I_\Delta+(\Theta)))_k \geq \binom d k \beta_{k-1}(\Delta)$
by Lemma \ref{2.1}(ii),
we have $\beta_k(\Delta)=0$ for $k \geq r$, proving (i).
Moreover, this fact and Lemmas \ref{2.2} and \ref{2.4} show 
$\mu_k(I_\Delta)=\mu_k(I_\Delta+(\Theta))$ for $k \geq r+1$.
Since $S/(I_\Delta+(\Theta))_{r+1}=0$ by \eqref{2-1},
the statement (ii) follows from Lemma \ref{2.3}.
\end{proof}

\section{Stacked manifolds with boundary}

In this section, we study $r$-stacked manifolds with boundary.
Recall that
a homology $d$-manifold with boundary is said to be {\em $r$-stacked} if it has no interior faces of dimension $\leq d-r-1$
and that a homology manifold without boundary is said to be {\em $r$-stacked} if it is the boundary of an $r$-stacked homology manifold with boundary.
For a simplicial complex $\Delta$ of dimension $d-1$, let
$$g_i(\Delta)= h_i(\Delta)-h_{i-1}(\Delta)$$
for $i=0,1,\dots,d$.

\subsection*{Enumerative criterion}
It is known that a homology ball $\Delta$ is $(r-1)$-stacked if and only if
$h_r(\Delta)=0$.
See \cite[Proposition 2.4]{Mc}.
We first extend this property for stacked manifolds.

Let $\Delta$ be a homology $(d-1)$-manifold with boundary.
Then the Dehn--Sommerville relations for homology manifolds with boundary \cite[Corollary 2.2]{Gr} say
\begin{align}
\label{d1}
g_i(\partial \Delta)
&= h_i(\Delta)-h_{d-i}(\Delta)+ \binom d i (-1)^{d-1-i} \widetilde \chi (\Delta)
\end{align}
where $\widetilde \chi(\Delta)= \sum_{k=-1}^{d-1} (-1)^k f_k(\Delta)$
is the reduced Euler characteristic.
By substituting $h_{d-i}(\Delta)= h_{d-i}''(\Delta) + \binom d i \sum_{k=1}^{d-i} (-1)^{d-i-k} \beta_{k-1}(\Delta)$ and
$\widetilde \chi(\Delta)=\sum_{k=0}^{d-1} (-1)^k \beta_k(\Delta)$
to \eqref{d1},
we obtain
\begin{align}
\label{d2}
g_i(\partial \Delta)
&= h_i(\Delta)-h''_{d-i}(\Delta)+ \binom d i \sum_{k=d-i}^{d-1} (-1)^{d-1-i-k} \beta_k(\Delta).
\end{align}

\begin{theorem}
\label{3.1}
Let $1 \leq r \leq d$ and let $\Delta$ be a homology $(d-1)$-manifold with boundary.
Then $\Delta$ is $(r-1)$-stacked if and only if $h_r''(\Delta)=0$.
\end{theorem}

\begin{proof}
We first prove that $\Delta$ is $(r-1)$-stacked if and only if $g_i(\partial \Delta)=h_i(\Delta)$ for all $i \leq d-r$.
Indeed, it is clear that $\Delta$ is $(r-1)$-stacked if and only if $f_i(\partial \Delta)=f_i(\Delta)$ for all $i \leq d-r-1$.
Consider the equations
\begin{align}
\label{hg2} \sum_{i=0}^{d} f_{i-1}(\Delta) t^i 
=\sum_{i=0}^{d} h_i(\Delta) t^i (t+1)^{d-i}
\end{align}
and
\begin{align}
\label{hg1}
\sum_{i=0}^{d-1} f_{i-1}(\partial \Delta) t^i 
&=\sum_{i=0}^{d-1} h_i(\partial \Delta) t^i (t+1)^{d-i-1}\\
\nonumber &=\sum_{i=0}^{d-1} h_i(\partial \Delta) \{ t^i (t+1)^{d-i} - t^{i+1} (t+1)^{d-(i+1)}\}.
\end{align}
By comparing the coefficients of the polynomials in \eqref{hg2} and \eqref{hg1},
we conclude that $f_i(\partial \Delta)=f_i(\Delta)$ for $i \leq d-r-1$ if and only if $h_i(\partial \Delta)-h_{i-1}(\partial \Delta)=h_i(\Delta)$ for all $i \leq d-r$.

We first prove the `if' part.
Suppose $h_r''(\Delta)=0$.
Then we have $h_k''(\Delta)=0$ for all $k \geq r$ (see (1) in the proof of Proposition \ref{2.5}).
Also, 
$\beta_r(\Delta)= \cdots = \beta_{d-1}(\Delta)=0$
by Proposition \ref{2.5}.
Then the Dehn--Sommerville relation \eqref{d2} shows
$$g_i(\partial \Delta)=h_i(\Delta)$$
for all $i \leq d-r$, as desired.

Next, we prove the `only if' part.
Suppose $g_i(\partial \Delta)=h_i(\Delta)$ for all $i \leq d-r$.
The Dehn--Sommerville relations \eqref{d2} imply
\begin{align}
\label{3a-1}
h_{d-i}''(\Delta)= - \binom d i \beta_{d-i}(\Delta) + \binom d i \sum_{k=d-i+1}^{d-1}(-1)^{d-1-i-k} \beta_k(\Delta)
\end{align}
for all $i \leq d-r$.
We show by induction on $i$ that $\beta_{d-i}(\Delta)=0$ and $h_{d-i}''(\Delta)=0$ for $i \leq d-r$:
The claim is clear for $i=0$ by \eqref{3a-1}.
For $i>0$, by induction the second summand on the right-hand side of \eqref{3a-1} vanish.
Thus $h_{d-i}''(\Delta)= - \binom d i \beta_{d-i}(\Delta)$.
Since $h''$-vectors and Betti numbers are non-negative we have $h_{d-i}''(\Delta)=\beta_{d-i}(\Delta)=0$.
\end{proof}

\subsection*{Vanishing of missing faces}
If $\Delta$ is an $(r-1)$-stacked triangulated ball then $\Delta$ is Cohen--Macaulay and $h_r(\Delta)=0$.
These facts and Lemmas \ref{2.2} and \ref{2.3} say that $\Delta$ has no missing faces of dimension $ \geq r$
(another proof of this fact was given in \cite[Lemma 2.10]{BD3}).
Proposition \ref{2.5} and Theorem \ref{3.1} prove an analogue of this fact for manifolds.

\begin{corollary}
\label{3.3}
Let $\Delta$ be an $(r-1)$-stacked homology manifold with boundary.
Then
\begin{itemize}
\item[(i)] $\beta_k(\Delta)=0$ for $k \geq r$.
\item[(ii)] $\Delta$ has no missing $k$-faces of dimension $\geq r+1$.
\end{itemize}
\end{corollary}

Finally,  we give a few known examples of stacked manifolds.

\begin{example}[K\"uhnel--Lassmann construction \cite{Ku,KL}]
\label{ex1}
Let $K_{d,n}$ be the simplicial complex on $[n]$ generated by the facets
$$\{ \{i,i+1,\dots,i+d-1\}: i=1,2,\dots,n\},$$
where $i+k$ means $i+k-n$ if $i+k >n$.
If $n \geq 2d-1$ then $K_{d,n}$ is a homology manifold whose boundary 
triangulates either $S^1 \times S^{d-3}$ or a non-orientable $S^{d-3}$-bundle over $S^1$ \cite{KL}.
Since the interior faces of $K_{d,n}$
are $\{i,i+1,\dots,i+d-2\}$ for $i=1,2,\dots,n$,
the simplicial complex $K_{d,n}$ is $1$-stacked and has the $h''$-vector $(1,n-d,0,\dots,0)$.
\end{example}

\begin{example}[Klee--Novik construction \cite{KN}]
Let $X=\{x_1,\dots,x_d\}$ and $Y=\{y_1,\dots,y_d\}$ be disjoint sets.
For integers $0 \leq i \leq d-2$, let $B_{d,i}$ be the simplicial complex on the vertex set $X \cup Y$
generated by the facets
$$ \{ \{ z_1,\dots,z_d\}: z_i \in \{x_i,y_i\} \mbox{ and } \# \{k: \{z_k,z_{k+1}\} \not \subset X \mbox{ and } \{z_k,z_{k+1}\} \not \subset Y \} \leq i\}.$$
The simplicial complex $B_{d,i}$ is a combinatorial manifold whose boundary triangulates $S^i \times S^{d-i-2}$
and its $h''$-vector is given by $h_k''(B_{d,i})= \binom d k$ for $k \leq i$
and $h''_{i+1}(B_{d,i})=0$ \cite[Proposition 5.1]{KN}.
In particular, these triangulated manifolds are $i$-stacked by Theorem \ref{3.1}.
\end{example}

\begin{remark}
\label{3.4}
If $\Delta$ is an $(r-1)$-stacked triangulated ball then $\Delta$ has no missing $r$-faces.
However, an $(r-1)$-stacked homology manifold with boundary could have missing $r$-faces.
Indeed, the simplicial complex $K_{4,7}$ in Example \ref{ex1} is $1$-stacked but has a missing face $\{1,4,7\}$.
\end{remark}

\section{Stacked manifolds without boundary}

In Sections 4 and 5,
we study $(r-1)$-stacked $(d-1)$-manifold without boundary with $r \leq \frac d 2$.
In this section, we study these manifolds from combinatorial viewpoints.

\subsection*{Uniqueness of stacked manifolds}
A homology $d$-manifold $\Delta$ with boundary is said to be a {\em ($\kk$-)homology $d$-ball} if
$\widetilde H_k(\Delta;\kk)=0$ for all $k$ and $\partial \Delta$ is a ($\kk$-)homology $(d-1)$-sphere.
For a simplicial complex $\Delta$ on $[n]$,
let
$$\Delta(r)=\{ F \subset [n]: \mathrm{skel}_r(F) \subset \Delta\},$$
where $\mathrm{skel}_r(F)=\{ G \subset F:  \#G \leq r+1\}$ is the {\em $r$-skeleton} of $F$.
This simplicial complex can be defined algebraically.
For a homogeneous ideal $I \subset S$,
let $I_{\leq k}$ be the ideal generated by all elements in $I$ of degree $\leq k$.
Then  it is easy to see that
$(I_\Delta)_{\leq r+1} = I_{\Delta(r)}$.

For an $(r-1)$-stacked homology $(d-1)$-sphere $\Delta$,
it was shown by McMullen \cite[Theorem 3.3]{Mc} (for polytopes) and by Bagchi and Datta \cite[Proposition 2.10]{BD} (for triangulated spheres)
that an $(r-1)$-stacked homology $d$-ball $\Sigma$ satisfying $\partial \Sigma=\Delta$ is unique.
Moreover, the following result was shown in  \cite [Corollary 3.6]{BD} (for polytopes) and in \cite[Lemma 2.1 and Theorem 2.3]{MN} (for homology spheres) by a different approach.

\begin{lemma}
\label{4.2}
Let $1 \leq r \leq \frac {d+1} 2$ and
$\Delta$ an $(r-1)$-stacked homology $(d-1)$-sphere.
If $\Sigma$ is an $(r-1)$-stacked homology $d$-ball with $\partial \Sigma=\Delta$ then $\Sigma=\Delta(r-1)$.
\end{lemma}

\begin{proof}
Observe that $\Sigma$ has no missing faces of dimension $ \geq r$ (see the discussion before Corollary \ref{3.3}).
Then we have $I_\Sigma=(I_\Sigma)_{\leq r}$.
Since $\Sigma$ and $\Delta$ have the same $(d-r)$-skeleton and $r -1 \leq d-r$,
we have $(I_\Sigma)_{\leq r} = (I_\Delta)_{\leq r}$.
Hence
$$
I_\Sigma=(I_\Sigma)_{\leq r} = (I_\Delta)_{\leq r}=I_{\Delta(r-1)},$$
which implies $\Sigma=\Delta(r-1)$.
\end{proof}

The following is an extension of Lemma \ref{4.2}.

\begin{theorem}
\label{4.3}
Let $1 \leq r  \leq \frac d 2$ and
$\Delta$ an $(r-1)$-stacked homology $(d-1)$-manifold without boundary.
If $\Sigma$ is an $(r-1)$-stacked homology $d$-manifold with $\partial \Sigma=\Delta$ then $\Sigma=\Delta(r)$.
\end{theorem}

\begin{proof}
Since $\Sigma$ is $(r-1)$-stacked,
by Corollary \ref{3.3}(ii), $\Sigma$ has no missing faces of dimension $\geq r+1$,
namely, $I_\Sigma=(I_\Sigma)_{\leq r+1}$.
Then the statement follows in the same way as in the proof of Lemma \ref{4.2}.
\end{proof}

\begin{remark}
We cannot replace $\Delta(r)$ by $\Delta(r-1)$ in Theorem \ref{4.3} by the same reason as in Remark \ref{3.4}.
Similarly, the  statement fails when $r=\frac {d+1} 2$.
\end{remark}

\subsection*{Vanishing of missing faces}
It was shown by Kalai \cite[Proposition 3.6]{Ka2} and Nagel \cite[Corollary 4.6]{Na} that if $\Delta$ is 
an $(r-1)$-stacked homology $(d-1)$-sphere and $r \leq \frac d 2$ then $\Delta$ has no missing $k$-faces for $r \leq k \leq d-r$
(they write statements only for polytopes but Nagel's proof works for homology spheres).
This fact can be generalized as follows.

\begin{theorem}
\label{4.1}
Let $1 \leq r < \frac d 2$ and let $\Delta$ be an $(r-1)$-stacked homology $(d-1)$-manifold without boundary.
Then
\begin{itemize}
\item[(i)] $\beta_k(\Delta)=0$ for $r \leq k \leq d-1-r$.
\item[(ii)] $\Delta$ has no missing $k$-faces with $ r+1 \leq k \leq d-r$.
\end{itemize}
\end{theorem}

\begin{proof}
Let $\Sigma$ be an $(r-1)$-stacked homology $d$-manifold with $\partial \Sigma = \Delta$.
Since $\Sigma$ and $\Delta$ have the same $(d-r)$-skeleton,
we have
$\beta_i(\Delta)=\beta_i(\Sigma)$ for  $i < d-r$
and $\mu_j(I_\Delta)=\mu_j( I_\Sigma)$ for $j \leq d-r+1$.
Then the statement follows from Corollary \ref{3.3}.
\end{proof}

\begin{remark}
\label{4.1.1}
Lemma \ref{4.2} and the above proof give another proof for the fact that
if $r \leq \frac d 2$ and if $\Delta$ is an $(r-1)$-stacked homology $(d-1)$-sphere
then $\Delta$ has no missing $k$-faces for $ r \leq k \leq d-r$.
\end{remark}

\subsection*{Local criterion}
Next, we discuss a local criterion of stackedness.
We say that a homology $d$-manifold without boundary is {\em locally $r$-stacked}
if all its vertex links are $r$-stacked.
It is clear from the definition that
if a homology manifold $\Delta$ is $r$-stacked then it is locally $r$-stacked.
It was shown by Kalai \cite[Proposition 3.5]{Ka2} that if $r< \frac d 2$ then the converse holds for the boundary of a simplicial $d$-polytope.
This property can be extended as follows:

\begin{theorem}
\label{4.4}
Let $1 \leq r< \frac d 2$.
Then a homology $(d-1)$-manifold without boundary is $(r-1)$-stacked if and only if it is locally $(r-1)$-stacked.
\end{theorem}

\begin{proof}
The `only if' part is obvious.
We prove the `if' part.
The proof is similar to that of \cite[Theorem 5.3]{Mc}.
Let $\Delta$ be a locally $(r-1)$-stacked homology $(d-1)$-manifold without boundary.
For every vertex $v$ of $\Delta$,
let $$D_v=\lk_\Delta(v)(r-1).$$
Then since $\lk_\Delta(v)$ is $(r-1)$-stacked, $D_v$ is a homology ball whose boundary is $\lk_\Delta(v)$ by Lemma \ref{4.2}.
Let
$$\Sigma = \bigcup_v v* D_v,$$
where $v*\Gamma$ denotes the cone of $\Gamma$ over a vertex $v$.
We claim that $\Sigma$ is an $(r-1)$-stacked homology manifold with $\partial \Sigma=\Delta$.

For any vertex $v$ of $\Delta$, we have
$$\lk_\Sigma(v)=D_v \cup \bigcup_{u \ne v} u* \lk_{D_u}(v).$$
Since $\lk_\Delta(v)$
and $\lk_\Delta(u)$ are $(r-1)$-stacked homology $(d-2)$-spheres,
they have no missing $r$-faces (see Remark \ref{4.1.1}).
Thus
$$\lk_{D_u}(v)=\lk_{\lk_\Delta(u)(r-1)}(v)=\lk_{\lk_\Delta(u)}(v)(r-1)$$
and
$$\lk_{D_v}(u) = \lk_{\lk_\Delta(v)(r-1)}(u)= \lk_{\lk_\Delta(v)}(u)(r-1).$$
Then since $\lk_{\lk_\Delta(v)}(u)=\lk_{\lk_\Delta(u)}(v)$,
we have $u* \lk_{D_u}(v)= u* \lk_{D_v}(u) \subset D_v$.
Hence
$$\lk_\Sigma(v)=D_v$$
for any vertex $v$ of $\Delta$,
which implies that $\Sigma$ is a homology manifold with boundary.
Also, one has
$$
F \in \partial \Sigma \Leftrightarrow
F-v \in \partial (\lk_\Sigma(v))=\partial D_v=\lk_\Delta(v) \mbox{ for some }v\in F
\Leftrightarrow F \in \Delta
$$
which implies $\partial \Sigma=\Delta$.
Finally, we show that $\Sigma$ is $(r-1)$-stacked: if $F \in \Sigma$ is an interior face then $F-v$ is an interior face of $\lk_\Sigma(v)=D_v$, where $v \in F$,
and we have $\dim F > d-r$ since $\lk_\Delta(v)$ is $(r-1)$-stacked.
\end{proof}

\begin{remark}
Theorem \ref{4.4} fails for $r= \frac d 2$.
Indeed, the join $\Delta$ of boundaries of two $r$-simplices is a $(2r-1)$-sphere
which is not $(r-1)$-stacked but it is locally $(r-1)$-stacked.
Indeed, $\Delta$ is not $(r-1)$-stacked since $\Delta(r-1)$ is the power set of $[n]$.
Also, $\Delta$ is locally $(r-1)$-stacked since, for every vertex $v$ of $\Delta$,
$\lk_\Delta(v)$ is the boundary of the join of an $(r-1)$-simplex and the boundary of an $r$-simplex.
\end{remark}

\begin{remark}
Theorems \ref{4.3} (for $r< \frac d 2$), \ref{4.1}(ii) and \ref{4.4} were also proved independently by Bagchi and Datta \cite[Theorem 2.19]{BD3}
with essentially the same method.
\end{remark}

\section{New necessary condition for face numbers of manifolds}

McMullen and Walkup \cite{MW} conjectured that,
for the boundary complex $\Delta$ of a simplicial $d$-polytope,
one has $h_{r-1}(\Delta) \leq h_r(\Delta)$ for $r \leq \frac d 2$ and if equality holds for some $r$
then $\Delta$ is $(r-1)$-stacked.
This conjecture is called the generalized lower bound conjecture (GLBC for short).
The first part of the GLBC was solved by Stanley \cite{Stg} in his proof of the necessity of the $g$-theorem
and the second part of the GLBC was recently proved in \cite{MN}.
Recall that a connected homology $d$-manifold $\Delta$ without boundary is said to be {\em orientable} if $\beta_d(\Delta)=1$.
Motivated by the GLBC, Bagchi and Datta \cite[Conjecture 1.6]{BD2} suggested the following conjecture.

\begin{conjecture}[{GLBC for triangulated manifolds}]
\label{5.1}
Let $\Delta$ be a connected triangulated $(d-1)$-manifold without boundary. Then
\begin{itemize}
\item[(i)] $h_r(\Delta) \geq h_{r-1}(\Delta)+ \binom {d+1} r  \sum_{j=1}^r (-1)^{r-j} \beta_{j-1}(\Delta)$ for $r=1,2,\dots,\lfloor \frac d 2 \rfloor$.
\item[(ii)] if an equality holds for some $r< \frac d 2$ in (i)
then  $\Delta$ is locally $(r-1)$-stacked.
\end{itemize}
\end{conjecture}

Concerning part (i) of the conjecture,
a similar conjecture was given by Swartz \cite{Sw}.
Moreover, it was proved by Novik and Swartz that (i) holds for all homology manifolds all whose
vertex links satisfy certain algebraic property called the weak
Lefschetz property.
See \cite[p.\ 270, Inequality (9)]{NS3}.
Also, the conjecture is known to be true for orientable manifolds when $r=2$ \cite[Theorem 5.2]{NS1}.

Conjecture \ref{5.1} suggests us to study the following invariant of simplicial complexes,
which we call the {\em $\tilde g$-vector}.
For a simplicial complex $\Delta$ of dimension $d-1$,
let 
$$\tilde g_r(\Delta)= h_r(\Delta) - h_{r-1}(\Delta) - \binom {d+1} r \sum_{j=1}^r (-1)^{r-j} \beta_{j-1}(\Delta)$$
for $r=0,1,2,\dots,\lfloor \frac d 2 \rfloor$, where $\tilde g_0(\Delta)=1$,
and let $\tilde g(\Delta)=(\tilde g_0(\Delta),\tilde g_1(\Delta),\dots,\tilde g_{\lfloor \frac d 2 \rfloor}(\Delta))$.
Then Conjecture \ref{5.1}(i) asks if $\tilde g_k(\Delta) \geq 0$ for all $k$ when $\Delta$ is a connected triangulated manifold without boundary.
For an $(r-1)$-stacked homology $(d-1)$-manifold with $r \leq \frac d 2$,
its $\tilde g$-vector has the following simple but interesting form.

\begin{proposition}
\label{5.7}
Let $1 \leq r \leq \frac d 2$ and let $\Delta$ be an $(r-1)$-stacked homology $d$-manifold with boundary.
Then
$\tilde g_i(\partial \Delta)=h_i''(\Delta)$
for $i \leq \frac d 2.$
\end{proposition}

\begin{proof}
Since $\Delta$ and $\partial \Delta$ have the same $\lfloor \frac d 2 \rfloor$-skeleton,
$\beta_{k-1}(\Delta)=\beta_{k-1}(\partial \Delta)$
for $k \leq \frac d 2$,
and,
as shown in the proof of Theorem \ref{3.1},
$$g_i(\partial \Delta)=h_i(\Delta)$$
for $i \leq \frac d 2$.
Subtracting $\binom {d+1} i \sum_{j=1}^{i} (-1)^{i-j} \beta_{j-1}(\Delta)$
from the above equation,
we obtain the desired equation.
\end{proof}

Recall that a vector $h=(h_0,h_1,\dots,h_t) \in \ZZ^{t+1}$ is said to be an {\em $M$-vector}
if there is a standard graded $\kk$-algebra $A$ such that $h_k=\dim_\kk A_k$ for $k=0,1,\dots,t$.
Lemma \ref{2.1}(ii) shows that, in Proposition \ref{5.7},
$\tilde g(\partial \Delta)$ is not only a non-negative vector but also an $M$-vector.
It is natural to ask if $\tilde g(\partial \Delta)$ is an $M$-vector for {\em any} homology manifold without boundary.
In this section, we prove that this property as well as Conjecture \ref{5.1} hold for orientable homology manifolds all whose links satisfy a certain algebraic condition described below.

We say that a homology $(d-1)$-sphere $\Delta$ on $[n]$ has the {\em weak Lefschetz property} (WLP for short)
if there is an l.s.o.p.\ $\Theta$ of $\kk[\Delta]=S/I_\Delta$  and a linear form $w \in S_1$ such that the multiplication
\begin{align}
\label{5-1}
\times w: (S/(I_\Delta+(\Theta)))_{i-1} \to (S/(I_\Delta+(\Theta)))_i 
\end{align}
is injective for $i \leq \frac {d+1} 2$ and is surjective for $i \geq \frac {d+1} 2$.
Note that it is known that the boundary complex of a simplicial polytope has the WLP over the rationals.
The following result is due to Swartz \cite[Theorem 4.26]{Sw}

\begin{lemma}[Swartz]
\label{5.3}
Let $\Delta$ be a connected orientable homology $(d-1)$-manifold without boundary on $[n]$.
Suppose that all the vertex links of $\Delta$ have the WLP.
Then there is an l.s.o.p.\ $\Theta$ of $\kk[\Delta]$ and a linear form $w$ such that the multiplication map
$$ \times w: (S/(I_\Delta+(\Theta)))_{i-1} \to (S/(I_\Delta+(\Theta)))_i 
$$
is surjective for all $i \geq \frac d 2 +1$.
\end{lemma}

The main result of this section is the following.

\begin{theorem}
\label{5.4}
With the same assumptions and notation as in Lemma \ref{5.3}, let $R=S/(I_\Delta+(\Theta))$ and $R'=R/wR$.
Then
\begin{itemize}
\item[(i)] there is an ideal $J \subset R'$ such that $\dim_\kk (R'/J)_i = \tilde g_i(\Delta)$ for $i \leq \frac d 2$.
In particular, $\tilde g(\Delta)$ is an $M$-vector.
\item[(ii)] if $\tilde g_r(\Delta)=0$ for some $r < \frac d 2$ then $\Delta$ is locally $(r-1)$-stacked.
\end{itemize}
\end{theorem}

Theorem \ref{5.4}(i) extends the result of Novik and Swartz \cite{NS3} who proved the non-negativity of $\tilde g$-vectors
for homology manifolds all whose vertex links have the WLP,
and Theorem \ref{5.4}(ii) proves that Conjecture \ref{5.1}(ii) holds for these manifolds.
In particular,
Conjecture \ref{5.1} holds 
for any rational homology manifold all whose vertex links are polytopal,
namely, are the boundary complexes of simplicial polytopes.
It was conjectured that any homology sphere has the WLP.
Thus, if this conjecture is true then Conjecture \ref{5.1} holds for all orientable homology manifolds.

We first introduce some lemmas which are used in the proof the above theorem.
The following was shown in \cite[Lemma 7.3]{No} and in \cite[Theorem 1.4]{NS2}.

\begin{lemma}
\label{5.5}
Let $\Delta$ be a connected orientable homology $(d-1)$-manifold without boundary on $[n]$.
\begin{itemize}
\item[(i)] (Novik) $h_i''(\Delta)=h_{d-i}''(\Delta)$ for all $i$.
\item[(ii)] (Novik--Swartz)
Let $\Theta=\theta_1,\dots,\theta_d$ be an l.s.o.p.\ of $\kk[\Delta]$,
$R=S/(I_\Delta+(\Theta))$ and $N=\bigoplus_{j=0}^{d-1} \Soc(R)_j$.
Then $F_{R/N}(t)= h_0''(\Delta)+h_1''(\Delta)t+\cdots +h_d''(\Delta)t^d$.
Moreover,
$R/N$ is Gorenstein.
\end{itemize}
\end{lemma}

\begin{corollary}
\label{5.5.5}
If $\Delta$ is a connected orientable homology $(d-1)$-manifold without boundary
then $\tilde g_i(\Delta) = h''_{d-i}(\Delta)-h'_{d-i+1}(\Delta)$ for $ i \leq \frac d 2$.
\end{corollary}

\begin{proof}
A routine computation shows
$$\tilde g_i (\Delta)= h_i''(\Delta)-h_{i-1}''(\Delta)- \binom d {i-1} \beta_{i-1}(\Delta).$$
On the other hand,
by Lemma \ref{5.5}(i) and the Poincar\'e duality $\beta_{d-i}(\Delta)=\beta_{i-1}(\Delta)$,
$h''_i(\Delta)=h_{d-i}''(\Delta)$ and
$$
h''_{i-1}(\Delta)+ \binom d {i-1} \beta_{i-1} (\Delta)
= h_{d-i+1}''(\Delta)+ \binom d {i-1} \beta_{d-i} (\Delta)
=  h_{d-i+1}'(\Delta).
$$
Hence we have
$\tilde g_i(\Delta)= h_{d-i}''(\Delta) - h_{d-i+1}'$.
\end{proof}

The next result about the stackedness was essentially  shown in \cite{MW} (necessity) and in \cite[Theorem 5.3]{MN} (sufficiency).

\begin{lemma}
\label{5.6}
Let $1 \leq r \leq \frac d 2$.
If a homology $(d-1)$-sphere $\Delta$ has the WLP then $\Delta$ is $(r-1)$-stacked if and only if $g_r(\Delta)=0$.
\end{lemma}

Now we prove Theorem \ref{5.4}.

\begin{proof}[Proof of Theorem \ref{5.4}]
(i)
Let $M^\vee$ be the Matlis dual of a graded $S$-module $M$ (c.f.\ \cite[\S 3.6]{BH}).
Note that $\dim_\kk M^\vee_j = \dim_\kk M_{-j}$ for all $j$.
Consider the short exact sequence ($N$ is as in Lemma \ref{5.5}(ii))
$$
0 \longrightarrow N \longrightarrow R \longrightarrow R/N \longrightarrow 0$$
and its Matlis dual
\begin{align}
\label{6-2}
0 \longleftarrow N^\vee \longleftarrow R^\vee \stackrel \phi \longleftarrow (R/N)^\vee \longleftarrow 0.
\end{align}
Let
$$A= \mathrm{Im}\ \! \phi \subset R^\vee.$$
Then, since $R/N$ is Gorenstein by Lemma \ref{5.5}(ii),
we have
\begin{align}
\label{6-3}
A \cong (R/N)^\vee \cong R/N (+d),
\end{align}
where $R/N(+d)$ denotes the graded module $R/N$ shifted in degree by $+d$.
Thus $(R/N(+d))_k= (R/N)_{k+d}$ for all $k$.

By Lemma \ref{5.3}, the multiplication $\times w : R_i \to R_{i+1}$ is surjective for $i \geq \frac d 2$.
Thus the multiplication
\begin{align}
\label{6-4}
\times w : R_{-i-1}^\vee \to R_{-i}^\vee
\end{align}
is injective for $i \geq \frac d 2$.
Since $\mideal N^\vee =0$,
we have $\mideal R^\vee \subset A$ by \eqref{6-2}.
Thus $w R^\vee \subset A$ is a submodule of $A$.
Then, by \eqref{6-3}, there is an ideal $I$ of $R$ such that
\begin{align*}
A/w R^\vee \cong R/(I+N) (+d).
\end{align*}
Observe that $I+N \supset w R$ since $w R^\vee \supset w A$.
We claim that the ideal $J=(I+N)/wR$ of $R'=R/wR$ satisfies the desired conditions.
Indeed,
since
$$R'/J \cong R/(I+N) \cong A/wR^\vee(-d)$$
and since $\dim_\kk (R/N)_i=h_i''(\Delta)=h''_{d-i}(\Delta)$
by Lemma \ref{5.5},
it follows from \eqref{6-3} and \eqref{6-4} that,
for $i \leq \frac d 2$,
$$\dim_\kk (R'/J)_i= \dim_\kk A_{i-d} - \dim_\kk R^\vee_{i-1-d}=h_{d-i}''(\Delta)-h_{d-i+1}'(\Delta)=\tilde g_i(\Delta),$$
where we use Corollary \ref{5.5.5} for the last equality.

(ii)
The proof is similar to that of \cite[Theorem 5.2]{NS1}.
Suppose $\tilde g_r(\Delta)=0$.
Since $\tilde g_r(\Delta)=h''_{d-r}(\Delta)-h'_{d-r+1}(\Delta)$,
the multiplication 
$$\times w : (R/N)_{d-r} \to R_{d-r+1}
$$
is bijective by Lemma \ref{5.3}.
For a vertex $v$ of $\Delta$,
let $\st_\Delta(v)=v* \lk_\Delta(v)$ be the star of $v$ in $\Delta$.
Since $\st_\Delta(v) \subset \Delta$,
by the Kind--Kleinschmidt condition \cite[III Lemma 2.4]{St},
$\Theta$ is an l.s.o.p.\ of $\kk[\st_\Delta(v)]=S/I_{\st_\Delta(v)}$ (regarding $\st_\Delta(v)$ as a simplicial complex on $[n]$).
Consider the commutative diagram
\begin{eqnarray*}
\begin{array}{ccc}
(S/(I_{\st_\Delta(v)} + (\Theta)))_{d-r} &\stackrel{\times x_v}{\longrightarrow} &R_{d-r+1}\medskip\\
\times w \uparrow  & & \uparrow \times w\medskip\\
(S/(I_{\st_\Delta(v)} + (\Theta)))_{d-r-1} &\stackrel{\times x_v}{\longrightarrow} &(R/N)_{d-r}.
\end{array}
\end{eqnarray*}
Since $S/(I_{\st_\Delta(v)} + (\Theta))$ is Gorenstein and has socle elements only in degree $d-1$,
by \cite[Proposition 4.24]{Sw} the horizontal maps in the above diagram are injective.
Then since the right vertical map is an isomorphism,
the left vertical map is injective.
This implies
$$ h_{d-r-1}(\lk_\Delta(v))=h_{d-r-1}(\st_\Delta(v))
\leq h_{d-r}(\st_\Delta(v))=h_{d-r}(\lk_\Delta(v)).$$
Since $\lk_\Delta(v)$ has the WLP by the assumption, its $h$-vector is unimodal.
Hence we have $h_{d-r-1}(\lk_\Delta(v))=h_{d-r}(\st_\Delta(v))$.
Then, by the Dehn--Sommerville equation,
it follows that $g_{r-1}(\lk_\Delta(v))=0$, and $\lk_\Delta(v)$ is $(r-1)$-stacked by Lemma \ref{5.6}.
\end{proof}

The local criterion for stackedness and Theorem \ref{5.4} imply the following criterion
for stackedness.

\begin{corollary}\label{6.9}
Let $r< \frac d 2$ and let $\Delta$ be a connected orientable homology $(d-1)$-manifold without boundary.
If all the vertex links of $\Delta$ have the WLP then $\Delta$ is $(r-1)$-stacked if and only if $\tilde g_r(\Delta) =0$.
\end{corollary}

\begin{proof}
The `if' part follows from Theorems \ref{4.4} and \ref{5.4}.
The `only if' part follows from Theorem \ref{3.1} and Proposition \ref{5.7}.
\end{proof}

We end this paper by a few questions.

\begin{conjecture}
\label{con1}
With the same assumptions and notation as in Theorem \ref{5.4},
$\dim_\kk (\Soc R')_r \geq \binom {d+1} r \beta_{r-1}(\Delta)$ for $r \leq \frac d 2.$
\end{conjecture}

If the conjecture is true, it will give a necessary condition for $h$-vectors of triangulated manifolds stronger than Theorem \ref{5.4}(i).
Indeed, Conjecture \ref{con1} implies Theorem \ref{5.4}(i)
since \cite[Theorem 3.2]{NS2} implies
$$ \dim_\kk R'_r= \dim_\kk R_r - \dim_\kk (R/\Soc(R))_{r-1} = h_r' - h_{r-1}''=\tilde g_r + \binom {d+1} r \beta_{r-1}$$ for $i \leq \frac d 2$.
For $(r-1)$-stacked homology $(d-1)$-manifolds without boundary with $r \leq \frac d 2$,
the conjecture follows from Lemma \ref{2.1}(ii) by taking $(\Theta,w)$ for $\Delta$ to be a general l.s.o.p.\ of $\kk[\Sigma]$,
where $\Sigma$ is the $(r-1)$-stacked homology manifold with $\partial \Sigma=\Delta$.
The conjecture also holds for triangulations of the product of spheres (under the WLP assumption)
since the ideal $J$ in Theorem \ref{5.4} is concentrated in a single degree in this case.

\begin{question}
Is it true that if $\Delta$ is a homology $(2k-1)$-manifold without boundary such that $\tilde g_k(\Delta)=0$
then $\Delta$ is $(k-1)$-stacked?
\end{question}

A similar question was raised by Novik--Swartz \cite[Problem 5.3]{NS1}
when $k=2$.
However, we do not have an answer even for this case.

\section*{Acknowledgments}
We would like to thank Isabella Novik for helpful comments on an earlier version of this paper.

\end{document}